\documentclass{elsarticle}%
\usepackage{amsfonts}
\usepackage{amsthm}
\usepackage{amsmath}
\usepackage{amssymb}
\usepackage{mathtools}
\usepackage{xcolor}
\usepackage{soul}
\usepackage{mathrsfs}
\usepackage{tabularx}
\usepackage[top=1in, bottom=1in, left=1in, right=1in]{geometry}
\usepackage{graphicx}%
\providecommand{\U}[1]{\protect\rule{.1in}{.1in}}
\newtheorem{theorem}{Theorem}[section]
\newtheorem{lemma}{Lemma}[section]

\theoremstyle{definition}
\newtheorem{definition}{Definition}[section]
\theoremstyle{remark}
\newtheorem{remark}{Remark}[section]

\numberwithin{equation}{section}
\begin{document}
	\begin{frontmatter}
		
		\title{Remarks on finite-approximate controllability of impulsive evolution systems via resolvent-like operator in Hilbert spaces}
		
		%%this line removes the date, but space is still left for it;
		%if used, remove the \vspace{-1cm}
		%\date{}
		
		%this gives the date in the form Mon 30 Jan 2012, 8:57pm;
		%if used, retain the \vspace{-1cm}
		%\date{\shortdayofweekname{\day}{\month}{\year}{ }\mydate\today}

		\author[1]{Javad A. Asadzade}
		\ead{javad.asadzade@emu.edu.tr}
            \author[1,2]{Nazim I. Mahmudov}
		\ead{nazim.mahmudov@emu.edu.tr}
		\cortext[cor1]{Corresponding author}

            \address[1]{Department of Mathematics, Eastern Mediterranean University, Mersin 10, 99628, T.R. North Cyprus, Turkey}
	\address[2]{Research Center of Econophysics, Azerbaijan State University of Economics (UNEC), Istiqlaliyyat Str. 6, Baku, 1001,  Azerbaijan}
	
		% Latex won't make the title unless told:
		%\maketitle
		
		%%to remove the space left for date, use:

		\begin{abstract}
            In this manuscript, we examine impulsive evolution systems in Hilbert spaces. Using a resolvent-like operator, we first establish the finite-approximate controllability for linear systems. Subsequently, by applying the Schauder fixed-point theorem (SFPT), we prove the existence of a solution and demonstrate the finite-approximate controllability of semilinear impulsive systems in Hilbert spaces. Finally, we extend these results to a broader application, specifically to the heat equation.
		\end{abstract}

  
	\end{frontmatter}
\section{Introduction}

Controllability is a cornerstone concept in control theory, playing a pivotal role in both deterministic and stochastic control of dynamical systems. It is essential for addressing various control challenges, including the irreducibility of transition semigroups, the stabilization of unstable systems through feedback mechanisms, and the formulation of optimal control strategies. A system is considered controllable if it can be steered to any desired state within a specified time using appropriate control inputs. Within control theory, several types of controllability are recognized, including exact, null, approximate, interior, velocity, boundary, and finite-approximate controllability, among others.

In practical applications, however, achieving exact controllability is often challenging. This has led to a growing focus on approximate and finite-approximate controllability in research. Both concepts relate to the ability to guide a system toward a desired state, but they differ in scope and approach. Specifically, \textbf{the distinction between approximate controllability and finite-approximate controllability} (\(A\)-controllability and \(F_A\)-controllability, respectively) lies in their treatment of time and precision. \underline{\(A\)-controllability} allows a system to be brought arbitrarily close to the desired state, permitting small deviations. Here, the system can approach the target state with increasing accuracy over time, although it may never reach the exact state. In contrast, \underline{\(F_A\)-controllability} requires the system to reach a state within a specified tolerance or precision in a finite amount of time. While both approaches tolerate slight inaccuracies, \(F_A\)-controllability is distinguished by its finite time constraint, as opposed to \(A\)-controllability’s potential for an indefinite time horizon.

Impulses and impulsive systems play a significant role in the study of dynamic systems. These systems are characterized by abrupt changes in their state at specific moments, referred to as impulses. Impulsive systems are highly relevant for modeling real-world phenomena where sudden transitions occur, such as impacts in mechanical systems, bursts of activity in biological systems, or instantaneous control actions in engineering. Understanding the controllability of impulsive systems is crucial for designing control and stabilization strategies to address these discontinuities effectively. Additionally, impulsive systems appear in diverse fields, including physics, biology, economics, and engineering, underscoring the importance of their study for solving practical problems. For a more detailed analysis of impulsive systems, readers may refer to the monograph \cite{laks}.

The concept of controllability in dynamic control systems ensures that, with appropriate control functions, a system can transition between any initial and target states within a finite time period. Recently, significant attention has been devoted to the \(A\)-controllability of impulsive differential equations (IDEs), where impulses at finite intervals influence the system's state. This area of study has seen substantial progress, as evidenced by the growing body of literature (e.g., \cite{guan1,guan2,han,laks,Leela,muni,pandit,xie,xie2,zhao,zhao1}). Notable contributions have been made by researchers such as George et al. \cite{geo}, Benzaid and Sznaier \cite{ben}, Muni and George \cite{muni}, Xie and Wang \cite{xie2}, Guan et al. \cite{guan1,guan2}, Han et al. \cite{han}, and Zhao and Sun \cite{zhao,zhao1}.

Recent years have also witnessed advancements in the \(A\)-controllability of both deterministic and stochastic impulsive systems (see \cite{Jeet,Sakthivel,Shukla,Vijayakumar,Wei,xie,xie2}). These studies have established sufficient conditions for the \(A\)-controllability of semilinear systems by employing the resolvent operator condition introduced in \cite{Bashirov,Mahmudov1}. This condition has proven particularly effective, especially when the corresponding linear system demonstrates \(A\)-controllability. 

The resolvent operator condition has become a widely utilized tool in research on the \(A\)-controllability of semilinear IDEs. In the absence of impulses, this condition aligns with the \(A\)-controllability of the linear component of the associated semilinear evolution control system, as emphasized in \cite{Bashirov,Mahmudov1}.

In parallel, there has been a growing interest in \(F_A\)-controllability. For example, Arora, Mohan, and Dabas \cite{Arora} explored the finite-approximate controllability of impulsive fractional functional evolution equations of order \(1 < \alpha < 2\) in Hilbert spaces. Their work also addressed the finite-approximate controllability of semilinear fractional evolution equations of the same order, incorporating finite delays and employing variational methods. Similarly, Mahmudov has made significant contributions to this area. In \cite{Mahm0}, he used Schauder’s fixed-point theorem to establish finite-approximate controllability for first-order abstract semilinear evolution equations. Subsequently, in \cite{Mahm2,Mahm3}, he extended these results by developing sufficient conditions for finite-approximate controllability in fractional systems and Sobolev-type fractional systems using the variational approach. Further advancements were made by Wang et al. \cite{Wang}, who investigated the finite-approximate controllability of semilinear fractional differential equations involving the Hilfer derivative, and by Mahmudov \cite{Mahm1}, who explored this property for abstract semilinear evolution equations through a resolvent-like operator.

In 2024, Mahmudov \cite{Mahmudov2} advanced this research by studying the \(A\)-controllability of the following impulsive system in a Hilbert space:

\begin{equation}\label{eq0}
	\begin{cases}
		z^{\prime}(t)=Az(t)+\Omega u(t), & t\in[0,b]\setminus\{t_{1},\dots,t_{p}\},\\
		\Delta z(t_{k+1})=B_{k+1}z(t_{k+1})+D_{k+1}v_{k+1},& k=0,\dots,p-1,\\
		z(0)=z_{0}.
	\end{cases}
\end{equation}

In this system, \(z(\cdot)\in H\) represents the state variable in a Hilbert space \(H\) with norm \(\Vert z\Vert=\sqrt{\langle z,z\rangle}\), and the control function \(u(\cdot)\) belongs to \(L^2([0, b], U)\), where \(U\) is another Hilbert space. Additionally, \(v_k \in U\) for \(k = 1, 2, \dots, p\). This study marked the first investigation of such impulsive structures in infinite-dimensional spaces. Mahmudov employed semigroups and impulsive operators to construct solutions and derived necessary and sufficient conditions for the approximate controllability of linear impulsive evolution equations using the concept of the impulsive resolvent operator.

Further progress in this area was made by Asadzade and Mahmudov, who studied the existence and optimal control of impulsive stochastic evolution systems \cite{Asadzade2}. They also extended their work to examine the approximate controllability of semilinear systems \cite{Asadzade3} and developed the fractional analog of linear impulsive evolution systems in Hilbert spaces \cite{Asadzade1}. Simultaneously, Gupta and Dabas investigated impulsive controls in non-autonomous impulsive integro-differential equations in Hilbert spaces \cite{Gupta}, thereby enriching the understanding of impulsive systems in infinite-dimensional contexts.

A natural and compelling question arises as to whether \(F_A\)-controllability results can be extended effectively to impulsive systems. Addressing this question would represent a significant breakthrough in both the theoretical understanding and practical application of impulsive systems, especially those characterized by abrupt changes in state or behavior.

This paper is devoted to investigating the \(F_A\)-controllability of a specific class of impulsive evolution systems within the rigorous mathematical framework of Hilbert spaces. Establishing \(F_A\)-controllability for these systems would not only validate the relevance of advanced controllability concepts in impulsive settings but also pave the way for new strategies in the stabilization, optimization, and control of systems exhibiting impulsive dynamics.

To this end, we consider the following impulsive evolution system:

\begin{equation}\label{eq1}
	\begin{cases}
		z^{\prime}(t)=Az(t)+\Omega u(t)+\mu(t,z(t)), & t\in[0,b]\setminus\{t_{1},\dots,t_{p}\},\\
		\Delta z(t_{k+1})=B_{k+1}z(t_{k+1})+D_{k+1}v_{k+1},& k=0,\dots,p-1,\\
		z(0)=z_{0}.
	\end{cases}
\end{equation}

Here \(z(\cdot) \in H \) is in a Hilbert space .
The control \( u(\cdot) \) is an element of \(L^2([0, b], U) \), where ( \(H,\|z\| = \sqrt{\langle z, z\rangle} \) ) and \(U \) are Hilbert spaces, and \(v_k \in U \) for \(k = 1, \ldots, p \).
In this setting, \(A \) acts as a generator of a $C_0$-semigroup \(S(t) \) of continuous linear operators in \(H \). The  operators are: \(\Omega \in L(U, H) \), \(B_k \in L(H, H) \), and \(D_k \in L(U, H) \).
At each point of discontinuity \( t_k \) (for \( k = 1, \ldots, p \) with \( 0 = t_0 < t_1 < t_2 < \cdots < t_n < t_{p+1} = b \)), the state variable undergoes a jump, defined by \( \Delta z(t_k) = z(t_k^+) - z(t_k^-) \). Here, \( z(t_k^\pm) = \lim_{h \to 0^\pm} z(t_k + h) \), assuming that \( z(t_k^-) = z(t_k) \).
For operator compositions, \( \prod_{j=1}^{k} A_j \) represents the sequence \( A_1,  \ldots, A_k \), while for \( j = k+1 \) to \( k \), \( \prod_{j=k+1}^{k} A_j = 1 \). Likewise, \( \prod_{j=k}^{1} A_j \) refers to the sequence \( A_k, A_{k-1}, \ldots, A_1 \), and \( \prod_{j=k}^{k+1} A_j = 1 \).

The structure of this paper is organized as follows:

$\bullet$ Section 2: This section provides the mathematical background, including definitions, lemmas, and essential information related to controllability. These foundational concepts are critical for understanding the subsequent developments in the paper. 

$\bullet$ Section 3: Using a resolvent-like operator, we establish $F_{A}$-controllability for linear systems. This section presents the rigorous proofs and results for linear impulsive systems.  

$\bullet$ Section 4: By employing the SFPT, we demonstrate the existence of a solution and prove the $F_{A}$-controllability of semilinear impulsive systems in Hilbert spaces. This extends the results of the linear case to a more general framework. 

$\bullet$ Section 5: The findings are further extended to a broader application, specifically to the heat equation. This showcases the practical applicability of the theoretical results established in earlier sections. 
 \section{Preliminaries}
Let \( D \) be a finite-dimensional subspace of \( H \), and let \( \pi_D \) denote the orthogonal projection from \( H \) onto \( D \). The control system is said to be $F_{A}$-controllable on the interval \([0, b]\), if for every initial state \( z_0 \in H \), every target state \( z_b \in H \), and every tolerance \( \varepsilon > 0 \), there exists a control \( u \in L^2([0, b], U) \), such that the mild solution \( z(t) \) of the Cauchy problem \eqref{eq1} satisfies the following conditions:

1. \(\| z(b) - h \| < \varepsilon\), ensuring that the final state \( z(b) \) is within a distance \( \varepsilon \) of the target state \( h \) in the norm of \( H \).

2. \(\pi_D z(b) = \pi_D h\), ensuring that the projection of \( z(b) \) onto the finite-dimensional subspace \( D \) exactly matches the projection of \( h \) onto \( D \).

These conditions encapsulate the idea that the system can be controlled to reach a state arbitrarily close to the target \( h \), both in terms of proximity in the full space \( H \) and alignment in the finite-dimensional subspace \( D \), within the finite time interval \([0, b]\).

Now, we define the mild solution of \eqref{eq1} as follows:
 \begin{definition} 
 For every $u\in L^{2}([0,b],H)$ a function $z\in PC([0,b],H)$ is called a mild solution of \eqref{eq1} if
\begin{align}\label{eq2}
	z(t)=
		\begin{cases}
			S(t)z(0)+\int_{0}^{t} S(t-s)\big[\Omega u(s)+\mu(s,z(s))\big]ds,\,  0\leq t\leq t_{1},\\
			\\
		S(t-t_{k})z(t^{+}_{k})+\int_{t_{k}}^{t} S(t-s)\big[\Omega u(s)+\mu(s,z(s))\big]ds,\,
        t_{k}<t\leq t_{k+1},\quad k=1,2,\dots, p,
		\end{cases}
	\end{align}
where
\begin{align}\label{eq3}
	&z(t^{+}_{k})=\prod_{j=k}^{1}(\mathcal{I}+B_{j})S(t_{j}-t_{j-1})z_{0}\nonumber\\
	+&\sum_{i=1}^{k}\prod_{j=k}^{i+1}(\mathcal{I}+B_{j})S(t_{j}-t_{j-1})
	(\mathcal{I}+B_{i})\int_{t_{i-1}}^{t_{i}}S(t_{i}-s)\Omega u(s)ds\nonumber\\
 +&\sum_{i=1}^{k}\prod_{j=k}^{i+1}(\mathcal{I}+B_{j})S(t_{j}-t_{j-1})
	(\mathcal{I}+B_{i})\int_{t_{i-1}}^{t_{i}}S(t_{i}-s)\mu(s,z(s))ds\\
	+&\sum_{i=2}^{k}\prod_{j=k}^{i}(\mathcal{I}+B_{j}) S(t_{j}-t_{j-1}) D_{i-1}v_{i-1}+D_{k}v_{k}.\nonumber
\end{align}
 \end{definition}
In such a way, we define the space \(PC([0,b],H)\) as the set of piecewise continuous functions on the interval \([0,b]\) taking values in a Hilbert space \(H\). A function \(z : [0, b] \to H\) is said to belong to \(PC([0,b],H)\) if it satisfies the following conditions:
\bigskip

1. \(z(t)\) is continuous on \([0, b] \setminus \{t_1, t_2, \dots, t_p\}\), where \(\{t_1, t_2, \dots, t_p\}\) is a finite set of points in \([0,b]\).
\bigskip

2. At each point \(t_k \in \{t_1, t_2, \dots, t_p\}\), the function \(z(t)\) has left-hand and right-hand limits, denoted as \(z(t_k^-)\) and \(z(t_k^+)\), respectively.
\bigskip

3. The function \(z(t)\) may exhibit a jump discontinuity at \(t_k\), represented by:
   \[
   \Delta z(t_k) = z(t_k^+) - z(t_k^-),
   \]
   where \(\Delta z(t_k)\) defines the magnitude of the abrupt change at \(t_k\).
\bigskip

The norm \( \|z\|_{PC} \) for \( z \in PC([0,b],H) \) is given by:
\[
\|z\|_{PC} = \sup_{t \in [0,b]} \|z(t)\|,
\]
where \(\|z(t)\|\) is the norm of \(z(t)\) in the Hilbert space \(H\). This norm represents the maximum amplitude of \(z(t)\) throughout the interval \([0,b]\), encompassing both continuous variations and sudden transitions. Formally, \(PC([0,b],H)\) can be expressed as:
\[
PC([0,b],H) = \{z : [0,b] \to H \mid z \text{ is continuous on } [0,b] \setminus \{t_1, t_2, \dots, t_p\}, \, \Delta z(t_k) \text{ exists for all } t_k\},
\]
equipped with the norm $\|z\|_{PC}$.

This space is significant in the study of impulsive systems, where the state \(z(t)\) can undergo abrupt changes at specific points in time. These abrupt changes, also known as impulses, often arise due to instantaneous external forces, control interventions, or natural phenomena. They are mathematically captured as jump discontinuities at designated points \(t_k\). 
\bigskip

 On the other hand, for the sake of simplicity, let us define the operators as follows:
\begin{align*}
    \Gamma^{b}_{t_{p}}&=\int_{t_{p}}^{b}S(b-s)\Omega\Omega^{*}S^{*}(b-s)ds,\quad \tilde{\Gamma}^{b}_{t_{p}}=S(b-t_{p})D_{p}D^{*}_{p}S^{*}(b-t_{p}),\end{align*}
    \begin{align*}
            \Theta^{t_{p}}_{0}&=S(b-t_{p})\sum_{i=1}^{p}\prod_{j=p}^{i+1}(\mathcal{I}+B_{j})S(t_{j}-t_{j-1})
	(\mathcal{I}+B_{i})
   \int_{t_{i-1}}^{t_{i}}S(t_{i}-s)\Omega\Omega^{*}S^{*}(t_{k}-s)\, ds\\
    &\times(\mathcal{I}+B^{*}_{i})\prod_{k=i+1}^{p}S^{*}(t_{k}-t_{k-1})(\mathcal{I}+B^{*}_{k})S^{*}(b-t_{p}),
\end{align*}
\begin{align*}
    \tilde{\Theta}^{t_{p}}_{0}&=S(b-t_{p})\sum_{i=2}^{p}\prod_{j=p}^{i}(\mathcal{I}+B_{j})S(t_{j}-t_{j-1})D_{i-1}D^{*}_{i-1}\prod_{k=i}^{p}S^{*}(t_{k}-t_{k-1})(\mathcal{I}+B^{*}_{k})S^{*}(b-t_{p}).
\end{align*}
Next, we present the following lemma, which is crucial for the subsequent sections of this study. The results derived here follow a similar methodology to the one used in \cite{Mahmudov1} for impulsive resolvent-like operators applied to \eqref{eq1}.

\begin{lemma}\label{q2}\cite{Mahmudov1}
    Let $(H,\Vert \cdot\Vert)$ be a Hilbert space. If $\Gamma: H\to H$ is a linear non-neqative operator, then the operator $\alpha(\mathcal{I}-\pi_{D})+\Gamma: H\to H$ is invertible and 
    \begin{align}\label{u1}
        \big\Vert (\alpha(\mathcal{I}-\pi_{D})+\Gamma)^{-1}h\big\Vert\leq \frac{1}{\min(\alpha,\delta)}\Vert h\Vert ,\, h\in H,
    \end{align}
    where $\delta=\min\{\langle \pi_{D}\Gamma\pi_{D}\varphi, \varphi\rangle: \Vert \pi_{D} \varphi \Vert=1\}$. Moreover, if $\Gamma: H\to H$ is a linear positive operator, then 
    
    \begin{align} \label{u2} 
    (\alpha (\mathcal{I}-\pi_{D})+\Gamma)^{-1}=\big( \mathcal{I}-\alpha(\alpha \mathcal{I}+\Gamma)^{-1}\pi_{D}\big)^{-1}(\alpha \mathcal{I}+\Gamma)^{-1}.
    \end{align}
\end{lemma}
	\section{$F_{A}$ contrability of linear system}
    
In this section, we explore the  $F_{A}$-controllability  of the following linear impulsive evolution system:
    \begin{equation}\label{eq4}
	\begin{cases}
		z^{\prime}(t)=Az(t)+\Omega u(t), & t\in [0,b]\setminus\{t_{1},\dots,t_{p}\},\\
		\Delta z(t_{k+1})=B_{k+1}z(t_{k+1})+D_{k+1}v_{k+1},& k=0,\dots,p-1,\\
		z(0)=z_{0}.
	\end{cases}
\end{equation}

We derive an explicit expression for the finite-approximating control using the impulsive resolvent-like operator \(\Big(\alpha (\mathcal{I}-\pi_{D})+\Gamma^{b}_{t_{p}}+ \tilde{\Gamma}^{b}_{t_{p}}+\Theta^{t_{p}}_{0}+\tilde{\Theta}^{t_{p}}_{0}\Big)^{-1}\).

The controllability Gramian is defined by

\[
MM^{*} = \Gamma^{b}_{t_{p}} + \tilde{\Gamma}^{b}_{t_{p}} + \Theta^{t_{p}}_{0} + \tilde{\Theta}^{t_{p}}_{0},
\]
where \(M\) is a bounded linear operator defined as follows:

\[
M : L^2([0, b], U) \times U^p \to H,
\]

and

\[\begin{aligned}
&\quad M(u(\cdot), \{v_k\}_{k=1}^p) = S(b - t_p) \bigg( \sum_{i=1}^p \prod_{j=p}^{i+1} (I + B_j) S(t_j - t_{j-1}) 
 (I + B_i)\int_{t_{i-1}}^{t_i} S(t_i - s) \Omega u(s) \, ds \bigg)\\
 &+ \int_{t_p}^b S(b - s) \Omega u(s) \, ds
+ S(b - t_p)  \sum_{i=2}^p \prod_{j=p}^i (I + B_j) S(t_j - t_{j-1}) D_{i-1}v_{i-1} + S(b - t_p) D_p v_p .
\end{aligned}\]
Before presenting the proof of the \( F_A \)-controllability results for linear systems, we introduce the following lemma, which plays a crucial role in the proof of the main result.

\begin{lemma}\label{q1}
Let \( (H, \|\cdot\|) \) be a Hilbert space. Assume that \(\Gamma^{b}_{t_{p}}+ \tilde{\Gamma}^{b}_{t_{p}}+\Theta^{t_{p}}_{0}+\tilde{\Theta}^{t_{p}}_{0} : H \to H \), is a linear positive operator.  

(a) For any sequence \( \{\alpha_n > 0\} \) converging to \( 0 \) as \( n \to \infty \), we have  
\[
\lim_{n \to \infty} \alpha_n (\alpha_n I + \Gamma^{b}_{t_{p}}+ \tilde{\Gamma}^{b}_{t_{p}}+\Theta^{t_{p}}_{0}+\tilde{\Theta}^{t_{p}}_{0})^{-1} \pi_D = 0.
\]  

(b) For any \( \alpha > 0 \), we have  
\[
\|\alpha (\alpha I + \Gamma^{b}_{t_{p}}+ \tilde{\Gamma}^{b}_{t_{p}}+\Theta^{t_{p}}_{0}+\tilde{\Theta}^{t_{p}}_{0})^{-1} \pi_D\| < 1.
\]  

(c) The operator \(\alpha (\alpha I + \Gamma^{b}_{t_{p}} + \tilde{\Gamma}^{b}_{t_{p}} + \Theta^{t_{p}}_{0} + \tilde{\Theta}^{t_{p}}_{0})^{-1} \pi_D\) is continuous with respect to \(\alpha > 0\).
\end{lemma}
\begin{proof}  
The proofs of parts \((a)\) and \((b)\) are similar to those given in \cite{Mahmudov1} and are therefore omitted here.

For the proof of part \((c)\), we proceed as follows:  

\begin{align*}  
    &\alpha_1 \big(\alpha_1 \mathcal{I} + \Gamma^{b}_{t_{p}} + \tilde{\Gamma}^{b}_{t_{p}} + \Theta^{t_{p}}_{0} + \tilde{\Theta}^{t_{p}}_{0}\big)^{-1} \pi_D  
    - \alpha \big(\alpha \mathcal{I} + \Gamma^{b}_{t_{p}} + \tilde{\Gamma}^{b}_{t_{p}} + \Theta^{t_{p}}_{0} + \tilde{\Theta}^{t_{p}}_{0}\big)^{-1} \pi_D \\  
    &= \alpha_1 \big(\alpha_1 \mathcal{I} + \Gamma^{b}_{t_{p}} + \tilde{\Gamma}^{b}_{t_{p}} + \Theta^{t_{p}}_{0} + \tilde{\Theta}^{t_{p}}_{0}\big)^{-1}  
    \big(\frac{1}{\alpha} - \frac{1}{\alpha_1}\big)\big(\Gamma^{b}_{t_{p}} + \tilde{\Gamma}^{b}_{t_{p}} + \Theta^{t_{p}}_{0} + \tilde{\Theta}^{t_{p}}_{0}\big) \\  
    &\quad \times \alpha \big(\alpha \mathcal{I} + \Gamma^{b}_{t_{p}} + \tilde{\Gamma}^{b}_{t_{p}} + \Theta^{t_{p}}_{0} + \tilde{\Theta}^{t_{p}}_{0}\big)^{-1} \pi_D \\  
    &= \big(\alpha_1 \mathcal{I} + \Gamma^{b}_{t_{p}} + \tilde{\Gamma}^{b}_{t_{p}} + \Theta^{t_{p}}_{0} + \tilde{\Theta}^{t_{p}}_{0}\big)^{-1} (\alpha_1 - \alpha)  
    \big(\Gamma^{b}_{t_{p}} + \tilde{\Gamma}^{b}_{t_{p}} + \Theta^{t_{p}}_{0} + \tilde{\Theta}^{t_{p}}_{0}\big) \\  
    &\quad \times \big(\alpha \mathcal{I} + \Gamma^{b}_{t_{p}} + \tilde{\Gamma}^{b}_{t_{p}} + \Theta^{t_{p}}_{0} + \tilde{\Theta}^{t_{p}}_{0}\big)^{-1} \pi_D.  
\end{align*}  

It follows that \(\alpha \big(\alpha \mathcal{I} + \Gamma^{b}_{t_{p}} + \tilde{\Gamma}^{b}_{t_{p}} + \Theta^{t_{p}}_{0} + \tilde{\Theta}^{t_{p}}_{0}\big)^{-1} \pi_D\) is continuous with respect to \(\alpha\). Specifically,  
\begin{align*}  
    &\big\Vert \alpha_1 \big(\alpha_1 \mathcal{I} + \Gamma^{b}_{t_{p}} + \tilde{\Gamma}^{b}_{t_{p}} + \Theta^{t_{p}}_{0} + \tilde{\Theta}^{t_{p}}_{0}\big)^{-1} \pi_D \\  
    &\quad - \alpha \big(\alpha \mathcal{I} + \Gamma^{b}_{t_{p}} + \tilde{\Gamma}^{b}_{t_{p}} + \Theta^{t_{p}}_{0} + \tilde{\Theta}^{t_{p}}_{0}\big)^{-1} \pi_D \big\Vert  
    \leq \frac{|\alpha_1 - \alpha|}{\alpha_1} \to 0 \quad \text{as} \quad \alpha_1 \to \alpha.  
\end{align*}  

Thus, the result is established.  
  
\end{proof}  

Finally, we present the proof of the theorem regarding the \(F_{A}\)-controllability of the system \eqref{eq4} as follows.
\begin{theorem}
    The following statements are equivalent:
\bigskip

$(i)$ The system \eqref{eq4} is approximately contrable on $[0,b]$.
\bigskip

$(ii)$ $\Gamma^{b}_{t_{p}}+ \tilde{\Gamma}^{b}_{t_{p}}+\Theta^{t_{p}}_{0}+\tilde{\Theta}^{t_{p}}_{0}$ is strictly positive.
\bigskip

$(iii)$ $\alpha \Big(\alpha \mathcal{I}+\Gamma^{b}_{t_{p}}+ \tilde{\Gamma}^{b}_{t_{p}}+\Theta^{t_{p}}_{0}+\tilde{\Theta}^{t_{p}}_{0}\Big)^{-1}$ converges to zero operator as $\alpha\to 0+$ in the strong operator topology.
\bigskip

$(iv)$ $\alpha \Big(\alpha (\mathcal{I}-\pi_{D})+\Gamma^{b}_{t_{p}}+ \tilde{\Gamma}^{b}_{t_{p}}+\Theta^{t_{p}}_{0}+\tilde{\Theta}^{t_{p}}_{0}\Big)^{-1}\to 0$ as $\alpha\to 0+$ in the strong operator topology.
\bigskip

$(v)$ The system $\eqref{eq4}$ is $F_{A}$-controllable on $[0,b]$.
\end{theorem}
\begin{proof}
   The equivalences \((i) \Leftrightarrow (ii) \Leftrightarrow (iii)\) are established in \cite{Mahmudov1} (see Theorem 13). To demonstrate the equivalence \((iii) \Leftrightarrow (v)\), for any \(\alpha > 0\) and \(h \in H\), we define the following functional \(\mathcal{J}_{\alpha}(\cdot, h): H \to \mathbb{R}\):  

\begin{align*}
\mathcal{J}_{\alpha}(\varphi, h) &= \frac{1}{2} \|M^* \varphi\|^2 + \frac{\alpha}{2} \big\langle(\mathcal{I}-\pi_{D})\varphi, \varphi \big\rangle
- \langle \varphi, h - S(b - t_p) \sum_{j=p}^1 (I + B_j) S(t_j - t_{j-1}) z_0 \rangle.
\end{align*}

Assuming that \((iii) \Leftrightarrow (ii)\) holds, it is evident that the functional \(\mathcal{J}_{\alpha}(\cdot, h)\) is Gateaux differentiable. Its derivative is given by:

\begin{align*}
    \frac{d}{d\varphi}\mathcal{J}_{\alpha}(\varphi, h)&=\Gamma^{b}_{t_{p}}\varphi+ \tilde{\Gamma}^{b}_{t_{p}}\varphi+\Theta^{t_{p}}_{0}\varphi+\tilde{\Theta}^{t_{p}}_{0}\varphi+\alpha (\mathcal{I}-\pi_{D})\varphi
    -h + S(b - t_p) \sum_{j=p}^1 (I + B_j) S(t_j - t_{j-1}) z_0,
\end{align*}

Since \(\Gamma^{b}_{t_{p}} + \tilde{\Gamma}^{b}_{t_{p}} + \Theta^{t_{p}}_{0} + \tilde{\Theta}^{t_{p}}_{0}\) is positive, the functional \(\mathcal{J}_{\alpha}(\cdot, h)\) is strictly convex. Therefore, it has a unique minimum, which can be determined as follows: 

\begin{align*}
    & \Gamma^{b}_{t_{p}}\varphi+ \tilde{\Gamma}^{b}_{t_{p}}\varphi+\Theta^{t_{p}}_{0}\varphi+\tilde{\Theta}^{t_{p}}_{0}\varphi+\alpha (\mathcal{I}-\pi_{D})\varphi
    -h + S(b - t_p) \sum_{j=p}^1 (I + B_j) S(t_j - t_{j-1}) z_0=0,\\
    & \varphi_{min}=-\Big(\alpha (\mathcal{I}-\pi_{D})+ \Gamma^{b}_{t_{p}}+ \tilde{\Gamma}^{b}_{t_{p}}+\Theta^{t_{p}}_{0}+\tilde{\Theta}^{t_{p}}_{0}\Big)^{-1} \bigg(S(b - t_p) \sum_{j=p}^1 (I + B_j) S(t_j - t_{j-1}) z_0 -h\bigg).
\end{align*}
It follows that for the control 

\begin{align}\label{eq5}
    u_{\alpha}(s)&=\bigg( \sum_{k=1}^{p}\Omega^{*}S^{*}(t_{k}-s)\prod_{i=k+1}^{p}S^{*}(t_{i}-t_{i-1}) S^{*}(b-t_{p})\chi_{(t_{k-1},t_{k})}+\Omega^{*}S^{*}(b-s)\chi_{(t_{p},b)}\bigg)\varphi_{min},
\end{align}
\begin{equation}\label{e16}
\begin{cases}
    v^{\alpha}_{p}=D^{*}_{p}S^{*}(b-t_{p})\varphi_{min},\\ v^{\alpha}_{k}=D^{*}_{k}\prod_{i=k}^{p}S^{*}(t_{i} -t_{i-1})(I +B^{*}_i)S^{*}(b-t_p)\varphi_{min},\\
        k=1,\dots, p-1,
    \end{cases}
\end{equation}
\begin{align}\label{k1}
    z_{\alpha}(b)-h&=\alpha (\mathcal{I}-\pi_{D})\Big(\alpha (\mathcal{I}-\pi_{D})+ \Gamma^{b}_{t_{p}}+ \tilde{\Gamma}^{b}_{t_{p}}+\Theta^{t_{p}}_{0}+\tilde{\Theta}^{t_{p}}_{0}\Big)^{-1}\nonumber\\
    &\times\bigg(S(b - t_p) \sum_{j=p}^1 (I + B_j) S(t_j - t_{j-1}) z_0 -h\bigg).
\end{align}
Thus,
\begin{align*}
    \lim_{\alpha\to 0+}\Vert z_{\alpha}(b)-h\Vert&=\lim_{\alpha\to 0+}\alpha\bigg\Vert (\mathcal{I}-\pi_{D})\Big(\alpha (\mathcal{I}-\pi_{D})+ \Gamma^{b}_{t_{p}}+ \tilde{\Gamma}^{b}_{t_{p}}+\Theta^{t_{p}}_{0}+\tilde{\Theta}^{t_{p}}_{0}\Big)^{-1}\\
    &\times \bigg(S(b - t_p) \sum_{j=p}^1 (I + B_j) S(t_j - t_{j-1}) z_0 -h\bigg)\bigg\Vert=0,
\end{align*}
\begin{align*}
    \pi_{D}(z_{\alpha}(b)-h)=0,
\end{align*}
that is, the system \eqref{eq4} is $F_{A}$-controllable on $[0,b]$. Hence, we have $(iii)\Leftrightarrow (v)$. The reverse implication is clear, as $F_{A}$-controllability directly implies approximate controllability. To demonstrate the equivalence $(iii) \Leftrightarrow (iv)$, assume that for any \( h \in H \),

\[
\lim_{\alpha \to 0^-} \alpha \big\Vert \big(\alpha \mathcal{I} + \Gamma^{b}_{t_{p}} + \tilde{\Gamma}^{b}_{t_{p}} + \Theta^{t_{p}}_{0} + \tilde{\Theta}^{t_{p}}_{0}\big)^{-1} h \big\Vert = 0.
\]
For every $h\in H$, by using the \eqref{u2}, we have
\begin{align}\label{u3}
    &\quad \Vert \alpha \big(\alpha (\mathcal{I}-\pi_{D})+\Gamma^{b}_{t_{p}} + \tilde{\Gamma}^{b}_{t_{p}} + \Theta^{t_{p}}_{0} + \tilde{\Theta}^{t_{p}}_{0}\big)^{-1}h\Vert\nonumber\\
    &\leq \bigg\Vert \big( \mathcal{I}-\alpha(\alpha\mathcal{I}+\Gamma^{b}_{t_{p}} + \tilde{\Gamma}^{b}_{t_{p}} + \Theta^{t_{p}}_{0} + \tilde{\Theta}^{t_{p}}_{0})^{-1} \pi_{D}\big)^{-1}\bigg\Vert\nonumber\\
    &\times\big\Vert \alpha (\alpha\mathcal{I}+\Gamma^{b}_{t_{p}} + \tilde{\Gamma}^{b}_{t_{p}} + \Theta^{t_{p}}_{0} + \tilde{\Theta}^{t_{p}}_{0})^{-1}h\big\Vert\nonumber\\
    &\leq \frac{1}{1-\big\Vert \alpha\big(\alpha\mathcal{I}+\Gamma^{b}_{t_{p}} + \tilde{\Gamma}^{b}_{t_{p}} + \Theta^{t_{p}}_{0} + \tilde{\Theta}^{t_{p}}_{0}\big)^{-1} \pi_{D}\big\Vert}\nonumber\\
     &\times\big\Vert \alpha (\alpha\mathcal{I}+\Gamma^{b}_{t_{p}} + \tilde{\Gamma}^{b}_{t_{p}} + \Theta^{t_{p}}_{0} + \tilde{\Theta}^{t_{p}}_{0})^{-1}h\big\Vert.
\end{align}

By using \eqref{u3}, and Lemma \ref{q1}, we have
\begin{align*}
    \beta=\max_{0\leq \alpha\leq 1}\Big\Vert \alpha\big(\alpha\mathcal{I}+\Gamma^{b}_{t_{p}} + \tilde{\Gamma}^{b}_{t_{p}} + \Theta^{t_{p}}_{0} + \tilde{\Theta}^{t_{p}}_{0}\big)^{-1} \pi_{D}\Big\Vert<1,
\end{align*}
\begin{align*}
   \Big\Vert \alpha \big(\alpha (\mathcal{I}-\pi_{D})+\Gamma^{b}_{t_{p}} + \tilde{\Gamma}^{b}_{t_{p}} + \Theta^{t_{p}}_{0} + \tilde{\Theta}^{t_{p}}_{0}\big)^{-1}h\Big\Vert
    \leq \frac{1}{1-\beta}\Big\Vert \alpha\big(\alpha\mathcal{I}+\Gamma^{b}_{t_{p}} + \tilde{\Gamma}^{b}_{t_{p}} + \Theta^{t_{p}}_{0} + \tilde{\Theta}^{t_{p}}_{0}\big)^{-1} h\Big\Vert.
\end{align*}
As \( \alpha \to 0^+ \), the term $
\alpha \Big(\alpha (\mathcal{I}-\pi_{D})+\Gamma^{b}_{t_{p}}+ \tilde{\Gamma}^{b}_{t_{p}}+\Theta^{t_{p}}_{0}+\tilde{\Theta}^{t_{p}}_{0}\Big)^{-1}$  approaches zero in the strong operator topology. The implication \( (iv) \Rightarrow (v) \) is a direct consequence of Equation \eqref{k1}.  
\end{proof}
\begin{remark}
    The control
    \begin{align*}
         u_{\alpha}(s)&=\bigg( \sum_{k=1}^{p}\Omega^{*}S^{*}(t_{k}-s)\prod_{i=k+1}^{p}S^{*}(t_{i}-t_{i-1}) S^{*}(b-t_{p})\chi_{(t_{k-1},t_{k})}+\Omega^{*}S^{*}(b-s)\chi_{(t_{p},b)}\bigg)\\
    &\times \Big(\alpha (\mathcal{I}-\pi_{D})+ \Gamma^{b}_{t_{p}}+ \tilde{\Gamma}^{b}_{t_{p}}+\Theta^{t_{p}}_{0}+\tilde{\Theta}^{t_{p}}_{0}\Big)^{-1} \bigg(h-S(b - t_p) \sum_{j=p}^1 (I + B_j) S(t_j - t_{j-1}) z_0\bigg)
    \end{align*}
    in addition to the requirement of approximately controllability provides finite-dimensional exact controllability of \eqref{eq4}.
\end{remark}
\section{$F_{A}$ contrability of semilinear system}
In this section, we demonstrate that for any \(\alpha > 0\) and for any \(h \in H\), the system described by \eqref{eq1} admits at least one mild solution. The corresponding control is given by:  
\begin{align}\label{f1}
    u_{\alpha}(s)&=\bigg( \sum_{k=1}^{p}\Omega^{*}S^{*}(t_{k}-s)\prod_{i=k+1}^{p}S^{*}(t_{i}-t_{i-1})S^{*}(b-t_{p})\chi_{(t_{k-1},t_{k})}+\Omega^{*}S^{*}(b-s)\chi_{(t_{p},b)}\bigg)\Psi,
\end{align}
where 
\begin{align*}
    \Psi&= \big(\alpha (\mathcal{I}-\pi_{D})+\Gamma^{b}_{t_{p}} + \tilde{\Gamma}^{b}_{t_{p}} + \Theta^{t_{p}}_{0} + \tilde{\Theta}^{t_{p}}_{0}\big)^{-1} \sigma,
\end{align*}
\begin{align*}
    \sigma&=h-S(b - t_p) \sum_{j=p}^1 (I + B_j) S(t_j - t_{j-1}) z_0
    -S(b - t_p)\sum_{i=1}^{p}\prod_{j=p}^{i+1}(\mathcal{I}+B_{j})S(t_{j}-t_{j-1})\\
	&\times (\mathcal{I}+B_{i})\int_{t_{i-1}}^{t_{i}}S(t_{i}-s)\mu(s,z(s))\,ds-\int_{t_{p}}^{b}S(b-s)\mu(s,z(s))ds.
\end{align*}
To examine the existence  of system \eqref{eq1}, we make the following assumptions:    

$(A_{1})$  \(S(t) (t > 0)\) is a compact operator.   

$(A_{2})$ The mapping \( \mu(t, \cdot): H \to H \) is continuous for all \( t \in [0, b] \), and for every \( z \in H \), the function \( \mu(\cdot, z): [0, b] \to H \) is strongly measurable.  

$(A_{3})$ There exists a positive, integrable function \( g \in L^\infty([0, b], [0, +\infty)) \) and a continuous, non-decreasing function \( \Lambda_\mu : [0, +\infty) \to (0, +\infty) \) such that for all \( (t, z) \in [0, b] \times H \), the following inequality holds:  
\[
\|\mu(t, z)\| \leq g(t)\Lambda_\mu(\|z\|), \quad \text{with} \quad \liminf_{r \to \infty} \frac{\Lambda_\mu(r)}{r} = d < +\infty.
\]  

$(A_{4})$ The operator \( \Omega : U \to H \) is linear and continuous, with \( M_\Omega = \|\Omega\| \).  

\begin{theorem}\label{t1}
Let Assumptions $(A_{1}-A_{4})$ hold true. Then, for arbitrary \( h \in H \)  and every \( \alpha > 0 \), the system \eqref{eq1} with the control \eqref{f1} has at least one mild solution on \( [0,b] \), provided that the following condition is satisfied:

\begin{align}\label{f5}
 d\Bigg\{\sum_{m=1}^{k}(1+M_{B})^{m}M_{S}^{m+1}M_{\Omega}b\frac{\tilde{M} M_{2}}{\alpha(1-\delta)}+\frac{\tilde{M}M_{2} b}{\alpha(1-\delta)}
+ \sum_{m=1}^{k}(1+M_{B})^{m}M_{S}^{m+1}b\Vert g\Vert_{L^{\infty}}+M_{S}b\Vert g\Vert_{L^{\infty}}\Bigg\}<1
\end{align}
where \( \delta=\big\Vert \alpha (\alpha\mathcal{I}+\Gamma^{b}_{t_{p}} + \tilde{\Gamma}^{b}_{t_{p}} + \Theta^{t_{p}}_{0} + \tilde{\Theta}^{t_{p}}_{0})^{-1}\pi_{D}\big\Vert_{L(H)}<1\).
\end{theorem}
\begin{proof} We define the set  
\(\mathcal{Q} = \{z \in PC([0, b], H) : z(0)=z_{0}\}\) equipped with the norm \(\|\cdot\|_{PC}\). Next, we consider the set  
\(\mathcal{B}_r = \{z \in \mathcal{Q} : \|z\|_{PC} \leq r\}\), for \(r>0\).  
\bigskip

For \(\alpha>0\), we introduce an operator \(G_{\alpha}: \mathcal{Q} \to \mathcal{Q}\) defined as  

\[
(G_{\alpha}z)(t)= 
\begin{cases}  
S(t)z(0)+\int_{0}^{t} S(t-s)\big[\Omega u_{\alpha}(s)+\mu(s,z(s))\big]ds,\,  0\leq t\leq t_{1},\\\\
S(t-t_{k})z(t^{+}_{k})+\int_{t_{k}}^{t} S(t-s)\big[\Omega u_{\alpha}(s)+\mu(s,z(s))\big]ds,\,
 t_{k}<t\leq t_{k+1},\ k=1,2,\dots,p,
\end{cases}
\]
where  
\begin{align*}
z(t^{+}_{k}) &= \prod_{j=k}^{1}(\mathcal{I}+B_{j})S(t_{j}-t_{j-1})z_{0}\\
&+ \sum_{i=1}^{k}\prod_{j=k}^{i+1}(\mathcal{I}+B_{j})S(t_{j}-t_{j-1})
(\mathcal{I}+B_{i})\int_{t_{i-1}}^{t_{i}}S(t_{i}-s)\Omega u_{\alpha}(s)ds\\
&+ \sum_{i=1}^{k}\prod_{j=k}^{i+1}(\mathcal{I}+B_{j})S(t_{j}-t_{j-1})
(\mathcal{I}+B_{i})\int_{t_{i-1}}^{t_{i}}S(t_{i}-s)\mu(s,z(s))ds\\
& + \sum_{i=2}^{k}\prod_{j=k}^{i}(\mathcal{I}+B_{j}) S(t_{j}-t_{j-1}) D_{i-1}v_{i-1} + D_{k}v_{k}.
\end{align*}
Here, \(u_{\alpha}(\cdot)\) is defined as in \eqref{f1}. By the definition of \(G_{\alpha}\), it is evident that establishing the existence of a mild solution to system \eqref{eq1} is equivalent to proving that the operator \(G_{\alpha}\) has a fixed point. We demonstrate that \(G_{\alpha}\) has a fixed point through the following steps.  

\textbf{Step 1:} 
We consider the inclusion \( G_{\alpha}(\mathcal{B}_r) \subset \mathcal{B}_r \) for some \( r = r(\alpha) > 0 \). Suppose, for the sake of contradiction, that our claim is false. Then, for any \( \alpha > 0 \) and every \( r > 0 \), there exists \( z_r(\cdot) \in \mathcal{B}_r \) such that \( \Vert (G_{\alpha} z_r)(t)\Vert_{H}>r \), for some \( t \in [0,b] \) (where \( t \) may depend on \( r \)). First, we compute
\begin{align*}
    \Vert \sigma\Vert_{H}&\leq \Vert h\Vert_{H}+\Vert S(b - t_p)\Vert_{H} \sum_{j=p}^1 (1 + \Vert B_j\Vert_{H} ) \Vert S(t_j - t_{j-1})\Vert_{H} \Vert z_0\Vert_{H}\\
    &+\Vert S(b - t_p)\Vert_{H} \sum_{i=1}^{p}\prod_{j=p}^{i+1}(1+\Vert B_{j}\Vert_{H})\Vert S(t_{j}-t_{j-1})\Vert_{H} \\
	&\times (1+\Vert B_{i}\Vert_{H})\int_{t_{i-1}}^{t_{i}}\Vert S(t_{i}-s)\mu(s,z(s))\Vert_{H} \,ds+ \int_{t_{p}}^{b}\Vert S(b-s)\mu(s,z(s))\Vert_{H}\, ds\\
    &\leq \Vert h\Vert_{H}+pM_{S}^{2}(1+M_{B})\Vert z_0\Vert_{H}+\sum_{k=1}^{p}(1+M_{B})^{k}M_{S}^{k}\int_{t_{k-1}}^{t_{k}}\Vert S(t_{k}-s)\mu(s,z(s))\Vert_{H}\, ds\\
    &+ \int_{t_{p}}^{b}\Vert S(b-s)\mu(s,z(s))\Vert_{H}\, ds\\
     &\leq \Vert h\Vert_{H}+pM_{S}^{2}(1+M_{B})\Vert z_0\Vert_{H}+\sum_{k=0}^{p}(1+M_{B})^{k}M_{S}^{k+1}b\Lambda(r) \Vert g\Vert_{L^{\infty}}\\
     &=M_{1}+M_{2}\Lambda(r)=\gamma,
\end{align*}
where
\begin{align*}
    M_{1}=\Vert h\Vert_{H}+pM_{S}^{2}(1+M_{B})\Vert z_0\Vert_{H},\,\, M_{2}=\sum_{k=0}^{p}(1+M_{B})^{k}M_{S}^{k+1}b \Vert g\Vert_{L^{\infty}}.
\end{align*}
Then, we have
\begin{align}\label{cont}
    \Vert u_{\alpha}(s)\Vert_{U}&=\bigg\Vert\bigg( \sum_{k=1}^{p}\Omega^{*}S^{*}(t_{k}-s)\prod_{i=k+1}^{p}S^{*}(t_{i}-t_{i-1})\nonumber\\
    &\times S^{*}(b-t_{p})\chi_{(t_{k-1},t_{k})}+\Omega^{*}S^{*}(b-s)\chi_{(t_{p},b)}\bigg)\Psi\bigg\Vert_{U}\nonumber\\
    &\leq \frac{\tilde{M}}{\alpha}\Vert \alpha\big(\alpha (\mathcal{I}-\pi_{D})+\Gamma^{b}_{t_{p}} + \tilde{\Gamma}^{b}_{t_{p}} + \Theta^{t_{p}}_{0} + \tilde{\Theta}^{t_{p}}_{0}\big)^{-1} \sigma\Vert_{H}\nonumber\\
    &\leq \frac{\tilde{M}}{\alpha}\big\Vert (\mathcal{I}-\alpha (\alpha\mathcal{I}+\Gamma^{b}_{t_{p}} + \tilde{\Gamma}^{b}_{t_{p}} + \Theta^{t_{p}}_{0} + \tilde{\Theta}^{t_{p}}_{0})^{-1}\pi_{D})^{-1}\big\Vert_{L(H)}\nonumber\\
    &\times \big\Vert \alpha(\alpha\mathcal{I}+\Gamma^{b}_{t_{p}} + \tilde{\Gamma}^{b}_{t_{p}} + \Theta^{t_{p}}_{0} + \tilde{\Theta}^{t_{p}}_{0})^{-1}\sigma\big\Vert_{H}\nonumber\\
     &\leq \frac{\tilde{M}}{\alpha}\frac{1}{1-\Vert \alpha (\alpha\mathcal{I}+\Gamma^{b}_{t_{p}} + \tilde{\Gamma}^{b}_{t_{p}} + \Theta^{t_{p}}_{0} + \tilde{\Theta}^{t_{p}}_{0})^{-1}\pi_{D}\Vert_{L(H)}}\Vert \sigma\Vert_{H}\nonumber\\
     &\leq  \frac{\tilde{M} \gamma}{\alpha(1-\delta)},
\end{align}
where
\begin{align*}
    \tilde{M}=M_{\Omega} \sum_{k=1}^{p+1} M_{S}^{k},\,\,
    \delta=\big\Vert \alpha (\alpha\mathcal{I}+\Gamma^{b}_{t_{p}} + \tilde{\Gamma}^{b}_{t_{p}} + \Theta^{t_{p}}_{0} + \tilde{\Theta}^{t_{p}}_{0})^{-1}\pi_{D}\big\Vert_{L(H)}<1.
\end{align*}
Considering \( t \in [0, t_1) \), and applying the estimate \eqref{cont} alongside Assumptions \((A_1 - A_4)\), we derive 
\begin{align}\label{79}
    r<\Vert (G_{\alpha}z_{r})(t)\Vert_{H}&\leq M_{S}\Vert z_{0}\Vert_{H}+M_{S}M_{\Omega}\int_{0}^{t}\Vert u_{\alpha}(s)\Vert_{U}ds+M_{S}t_{1}\Vert g\Vert_{L^{\infty}}\Lambda_{\mu}(r)\nonumber\\
    &\leq M_{S}\Vert z_{0}\Vert_{H}+\frac{\tilde{M}\gamma t_{1}}{\alpha(1-\delta)}+M_{S}t_{1}\Vert g\Vert_{L^{\infty}}\Lambda_{\mu}(r)\nonumber\\
     &\leq M_{S}\Vert z_{0}\Vert_{H}+\frac{\tilde{M}\gamma b}{\alpha(1-\delta)}+M_{S}b\Vert g\Vert_{L^{\infty}}\Lambda_{\mu}(r).
\end{align}
For \(t\in [t_{k},t_{k+1}\, \text{for}\, k=1,\dots,p\), we have
\begin{align*}
    r<\Vert (G_{\alpha}z_{r})(t)\Vert_{H}&\leq M_{S}\Vert z_{r}(t^{+}_{k})\Vert_{H}+\frac{\tilde{M}\gamma (b-t_{k})}{\alpha(1-\delta)}+M_{S}(b-t_{k})\Vert g\Vert_{L^{\infty}}\Lambda_{\mu}(r)\\
    &\leq M_{S}\Vert z_{r}(t^{+}_{k})\Vert_{H}+\frac{\tilde{M}\gamma b}{\alpha(1-\delta)}+M_{S}b\Vert g\Vert_{L^{\infty}}\Lambda_{\mu}(r),
\end{align*}
where
\begin{align*}
   &\quad \Vert z_{r}(t_{k}^{+})\Vert_{H}\leq (1+M_{B})^{k}M^{k}_{S}\Vert z_{0}\Vert_{H}+\sum_{m=1}^{k}(1+M_{B})^{m}M_{S}^{m}M_{\Omega}b\Vert u_{\alpha}\Vert_{H}\\
    &+ \sum_{m=1}^{k}(1+M_{B})^{m}M_{S}^{m}b\Vert g\Vert_{L^{\infty}}\Lambda_{\mu}(r)+\sum_{m=1}^{k-1}(1+M_{B})^{m}M_{S}M_{D}M_{V}+M_{D}M_{V}\\
    &\leq (1+M_{B})^{k}M^{k}_{S}\Vert z_{0}\Vert_{H}+\sum_{m=1}^{k}(1+M_{B})^{m}M_{S}^{m}M_{\Omega}b\frac{\tilde{M} \gamma}{\alpha(1-\delta)}\\
    &+ \sum_{m=1}^{k}(1+M_{B})^{m}M_{S}^{m}b\Vert g\Vert_{L^{\infty}}\Lambda_{\mu}(r)+\sum_{m=1}^{k-1}(1+M_{B})^{m}M_{S}M_{D}M_{V}+M_{D}M_{V}.
\end{align*}
As a result, we get
\begin{align}\label{78}
    r&<\Vert (G_{\alpha}z_{r})(t)\Vert_{H} \leq (1+M_{B})^{k}M^{k+1}_{S}\Vert z_{0}\Vert_{H}+\sum_{m=1}^{k}(1+M_{B})^{m}M_{S}^{m+1}M_{\Omega}b\frac{\tilde{M} \gamma}{\alpha(1-\delta)}\nonumber\\
    &+ \sum_{m=1}^{k}(1+M_{B})^{m}M_{S}^{m+1}b\Vert g\Vert_{L^{\infty}}\Lambda_{\mu}(r)+\sum_{m=1}^{k-1}(1+M_{B})^{m}M^{2}_{S}M_{D}M_{V}+M_{S}M_{D}M_{V}\nonumber\\
    &+\frac{\tilde{M}\gamma b}{\alpha(1-\delta)}+M_{S}b\Vert g\Vert_{L^{\infty}}\Lambda_{\mu}(r)=M_{3}+M_{4}\Lambda_{\mu}(r)
\end{align}
where
\begin{align*}
    M_{3}&=(1+M_{B})^{k}M^{k+1}_{S}\Vert z_{0}\Vert_{H}+\sum_{m=1}^{k}(1+M_{B})^{m}M_{S}^{m+1}M_{\Omega}b\frac{\tilde{M} M_{1}}{\alpha(1-\delta)}\\
    &+\sum_{m=1}^{k-1}(1+M_{B})^{m}M^{2}_{S}M_{D}M_{V}+M_{S}M_{D}M_{V}+\frac{\tilde{M}M_{1} b}{\alpha(1-\delta)},
\end{align*}
\begin{align*}
    M_{4}&=\sum_{m=1}^{k}(1+M_{B})^{m}M_{S}^{m+1}M_{\Omega}b\frac{\tilde{M} M_{2}}{\alpha(1-\delta)}+\frac{\tilde{M}M_{2} b}{\alpha(1-\delta)}\\
    &+ \sum_{m=1}^{k}(1+M_{B})^{m}M_{S}^{m+1}b\Vert g\Vert_{L^{\infty}}+M_{S}b\Vert g\Vert_{L^{\infty}}.
\end{align*}
Using the $(A_{3})$, we have

\[
 \liminf_{r \to \infty} \frac{\Lambda_\mu(r)}{r} = d < +\infty.
\] 
Thus, by dividing by \( r \) in expressions \eqref{79} and \eqref{78}, and subsequently taking the limit as \( r \to \infty \), we arrive 
\begin{align*}
 d\Bigg\{\sum_{m=1}^{k}(1+M_{B})^{m}M_{S}^{m+1}M_{\Omega}b\frac{\tilde{M} M_{2}}{\alpha(1-\delta)}+\frac{\tilde{M}M_{2} b}{\alpha(1-\delta)}+ \sum_{m=1}^{k}(1+M_{B})^{m}M_{S}^{m+1}b\Vert g\Vert_{L^{\infty}}+M_{S}b\Vert g\Vert_{L^{\infty}}\Bigg\}>1.
\end{align*}
This leads to a contradiction with \eqref{f5}. Hence, for some \( r > 0 \), \( G_{\alpha}(\mathcal{B}_r) \subset \mathcal{B}_r \).

\textbf{Step 2:}
The operator \( G_\alpha \) is compact, and this is demonstrated using the infinite-dimensional version of the generalized Arzelà–Ascoli theorem (see Theorem 2.1 in \cite{Wei}). According to this theorem, it suffices to establish the following:  

\bigskip
1. The image of \( \mathcal{B}_r \) under \( G_\alpha \) is uniformly bounded (addressed in Step (1)).  

\bigskip
2. \( (G_\alpha x)(t) \), for any \( x \in \mathcal{B}_r \), is equicontinuous for each \( t \in (t_j, t_{j+1}) \), where \( j = 0, 1, \dots, n \). 

\bigskip
3. The sets \( W(t) = \{(G_\alpha x)(t) : x \in \mathcal{B}_r, t \in J \setminus \{t_1, \dots, t_n\}\} \), \( W(t_j^+) = \{(G_\alpha x)(t_j^+) : x \in \mathcal{B}_r\} \), and \( W(t_j^-) = \{(G_\alpha x)(t_j^-) : x \in \mathcal{B}_r\} \) for \( j = 1, \dots, n \), are relatively compact.  

\bigskip
We first establish the equicontinuity of \( (G_\alpha x)(t) \) for each \( t \in (t_j, t_{j+1}) \), where \( j = 0, 1, \dots, n \). For this purpose, consider \( a_1, a_2 \in (0, t_1) \) with \( a_1 < a_2 \), and let \( x \in \mathcal{B}_r \). We then proceed to evaluate the following expression:
\begin{align}\label{48}
     \Vert (G_{\alpha}z)(a_{2})-(G_{\alpha}z)(a_{1})\Vert_{H}&\leq \Vert S(a_{2})-S(a_{1})\Vert_{L(H)}\Vert z(0)\Vert_{H}\nonumber\\
    &+\bigg\Vert \int_{a_{1}}^{a_{2}} S(a_{2}-s)\Omega u_{\alpha}(s)ds\bigg\Vert_{H}\nonumber\\
   & +\bigg\Vert \int_{a_{1}}^{a_{2}} S(a_{2}-s)\mu(s,z(s))ds\bigg\Vert_{H}\nonumber\\
    &+\bigg\Vert \int_{0}^{a_{1}} (S(a_{2}-s)-S(a_{1}-s))\Omega u_{\alpha}(s)ds\bigg\Vert_{H}\nonumber\\
    &+\bigg\Vert \int_{0}^{a_{1}} (S(a_{2}-s)-S(a_{1}-s))\mu(s,z(s))ds\bigg\Vert_{H}\nonumber\\
    &\leq \Vert S(a_{2})-S(a_{1})\Vert_{L(H)}\Vert z(0)\Vert_{H}\nonumber\\
    &+ \frac{M_{S}M_{\Omega}\tilde{M} \gamma}{\alpha(1-\delta)}(a_{2}-a_{1})+M_{S}(a_{2}-a_{1})\Vert g\Vert_{L^{\infty}}\Lambda_{\mu}(r)\nonumber\\
    &+\frac{M_{\Omega}\tilde{M} \gamma}{\alpha(1-\delta)}\int_{0}^{a_{1}}\Vert S(a_{2}-s)-S(a_{1}-s)\Vert_{L(H)}ds\nonumber\\
    &+\Vert g\Vert_{L^{\infty}}\Lambda_{\mu}(r)\int_{0}^{a_{1}}\Vert S(a_{2}-s)-S(a_{1}-s)\Vert_{L(H)}ds\nonumber\\
     &\leq \Vert S(a_{2})-S(a_{1})\Vert_{L(H)}\Vert z(0)\Vert_{H}\nonumber\\
    &+ \frac{M_{S}M_{\Omega}\tilde{M} \gamma}{\alpha(1-\delta)}(a_{2}-a_{1})+M_{S}(a_{2}-a_{1})\Vert g\Vert_{L^{\infty}}\Lambda_{\mu}(r)\nonumber\\
     &+\frac{M_{\Omega}\tilde{M} \gamma}{\alpha(1-\delta)a_{1}}\sup_{s\in[0,a_{1}]}\Vert S(a_{2}-s)-S(a_{1}-s)\Vert_{L(H)}\nonumber\\
    &+\Vert g\Vert_{L^{\infty}}\Lambda_{\mu}(r)a_{1}\sup_{s\in[0,a_{1}]}\Vert S(a_{2}-s)-S(a_{1}-s)\Vert_{L(H)}.
\end{align}
In a similar manner, for \( a_1, a_2 \in (t_k, t_{k+1}) \) with \( k = 1, \dots, p \), \( a_1 < a_2 \), and \( z \in \mathcal{B}_r \), the following computation can be performed:
\begin{align*}
     \Vert (G_{\alpha}z)(a_{2})-(G_{\alpha}z)(a_{1})\Vert_{H}&\leq \Vert S(a_{2}-t_{k})-S(a_{1}-t_{k})\Vert_{L(H)}\Vert z(t^{+}_{k})\Vert_{H}\\
      &+\bigg\Vert \int_{a_{1}}^{a_{2}} S(a_{2}-s)\Omega u_{\alpha}(s)ds\bigg\Vert_{H}\\
    &+\bigg\Vert \int_{a_{1}}^{a_{2}} S(a_{2}-s)\mu(s,z(s))ds\bigg\Vert_{H}\\
    &+\bigg\Vert \int_{t_{k}}^{a_{1}} (S(a_{2}-s)-S(a_{1}-s))\Omega u_{\alpha}(s)ds\bigg\Vert_{H}\\
    &+\bigg\Vert \int_{t_{k}}^{a_{1}} (S(a_{2}-s)-S(a_{1}-s))\mu(s,z(s))ds\bigg\Vert_{H}.
\end{align*}
Such that, we defined the following notation 
\begin{align*}
   &\quad  \Vert z(t^{+}_{k})\Vert_{H}\leq (1+M_{B})^{k}M^{k}_{S}\Vert z_{0}\Vert_{H}
    +\sum_{m=1}^{k}(1+M_{B})^{m}M_{S}^{m}M_{\Omega}b\frac{\tilde{M} \gamma}{\alpha(1-\delta)}\\
    &+ \sum_{m=1}^{k}(1+M_{B})^{m}M_{S}^{m}b\Vert g\Vert_{L^{\infty}}\Lambda_{\mu}(r)
    +\sum_{m=1}^{k-1}(1+M_{B})^{m}M_{S}M_{D}M_{V}+M_{D}M_{V}=D.
\end{align*}
Then, we obtain
\begin{align}\label{49}
     \Vert (G_{\alpha}z)(a_{2})-(G_{\alpha}z)(a_{1})\Vert_{H}& \leq \Vert S(a_{2}-t_{k})-S(a_{1}-t_{k})\Vert_{L(H)}D\nonumber\\
     &+ \frac{M_{S}M_{\Omega}\tilde{M} \gamma}{\alpha(1-\delta)}(a_{2}-a_{1})+M_{S}(a_{2}-a_{1})\Vert g\Vert_{L^{\infty}}\Lambda_{\mu}(r)\nonumber\\
    &+\frac{M_{\Omega}\tilde{M} \gamma}{\alpha(1-\delta)}\int_{t_{k}}^{a_{1}}\Vert S(a_{2}-s)-S(a_{1}-s)\Vert_{L(H)}ds\nonumber\\
    &+\Vert g\Vert_{L^{\infty}}\Lambda_{\mu}(r)\int_{t_{k}}^{a_{1}}\Vert S(a_{2}-s)-S(a_{1}-s)\Vert_{L(H)}ds\nonumber\\
   & \leq \Vert S(a_{2}-t_{k})-S(a_{1}-t_{k})\Vert_{L(H)}D\nonumber\\
    &+ \frac{M_{S}M_{\Omega}\tilde{M} \gamma}{\alpha(1-\delta)}(a_{2}-a_{1})+M_{S}(a_{2}-a_{1})\Vert g\Vert_{L^{\infty}}\Lambda_{\mu}(r)\nonumber\\
     &+\frac{M_{\Omega}\tilde{M} \gamma}{\alpha(1-\delta)a_{1}}\sup_{s\in[t_{k},a_{1}]}\Vert S(a_{2}-s)-S(a_{1}-s)\Vert_{L(H)}\nonumber\\
    &+\Vert g\Vert_{L^{\infty}}\Lambda_{\mu}(r)a_{1}\sup_{s\in[t_{k},a_{1}]}\Vert S(a_{2}-s)-S(a_{1}-s)\Vert_{L(H)}.
\end{align}

Since \( S(t) \) is compact for all \( t > 0 \), and  \( A \) generates a bounded \( C_0 \)-semigroup, then \( S(t) \) is uniformly continuous on \( [0, b] \). The compactness ensures that the semigroup smooths out irregularities in behavior as \( t \to 0^+ \). These properties imply that the right-hand sides of \eqref{48} and \eqref{49} approach zero as \( |a_2 - a_1| \to 0 \). Consequently, the image of \( \mathcal{B}_r \) under \( G_\alpha \) is equicontinuous for each \( t \in (t_k, t_{k+1}) \), where \( k = 0, 1, \dots, p \).  

We now show that the set \( W(t) = \{(G_\alpha z)(t) : z \in \mathcal{B}_r\} \), for each \( t \in J \), is relatively compact in \( \mathcal{B}_r \). The relative compactness of \( W(t) \) results from the fact that the operator \( S(t) \) is compact for \( t \geq 0 \). Thus, for any \( t \in J \), the set \( W(t) = \{(G_\alpha z)(t) : z \in \mathcal{B}_r\} \) is relatively compact.  

Therefore, by applying the generalized Arzelà–Ascoli theorem, we conclude that the operator \( G_\alpha \) is compact.
\bigskip

\textbf{Step 3:}
 To demonstrate the continuity of the operator \( G_{\alpha} \), consider a sequence \( \{z^{m}\}_{m=1}^\infty \subseteq \mathcal{B}_{r} \) such that \( z^{m} \to z \) in \( \mathcal{B}_{r} \). In other words, 

\[
\lim_{m \to \infty} \Vert z^{m} - z \Vert_{PC([0,b],H)} = 0.
\]
  From the above convergence with Assumption \((A_2)\), it follows directly that
 have
\begin{align}\label{m1}
    \Vert \mu(s,z^{m}(s))-\mu(s,z(s))\Vert_{H}\to 0,\,\, \text{as}\,\, m\to \infty.
\end{align}
Applying the convergence $\eqref{m1}$ and the dominated convergence theorem,  we obtain
\begin{align*}
      \Vert \sigma^{m}-\sigma\Vert_{H}&\leq \Vert S(b - t_p)\Vert_{H} \sum_{i=1}^{p}\prod_{j=p}^{i+1}(1+\Vert B_{j}\Vert_{H})\Vert S(t_{j}-t_{j-1})\Vert_{H} \\
	&\times (1+\Vert B_{i}\Vert_{H})\int_{t_{i-1}}^{t_{i}}\Vert S(t_{i}-s)(\mu(s,z^{m}(s))-\mu(s,z(s)))\Vert_{H} ds\\
    &+\int_{t_{p}}^{b}\Vert S(b-s)(\mu(s,z^{m}(s))-\mu(s,z(s)))\Vert_{H} ds\\
   &\leq \sum_{k=0}^{p}(1+M_{B})^{k}M_{S}^{k+1}\int_{t_{k-1}}^{t_{k}}\Vert \mu(s,z^{m}(s))-\mu(s,z(s))\Vert_{H} ds\\
   &\to 0,\, m\to \infty.
\end{align*}
Such that, we have
\begin{align*}
   & \quad \Vert \big(\alpha (\mathcal{I}-\pi_{D})+\Gamma^{b}_{t_{p}} + \tilde{\Gamma}^{b}_{t_{p}} + \Theta^{t_{p}}_{0} + \tilde{\Theta}^{t_{p}}_{0}\big)^{-1} (\sigma^{m}-\sigma)\Vert_{H}\\
    &=\frac{1}{\alpha}\Vert \alpha\big(\alpha (\mathcal{I}-\pi_{D})+\Gamma^{b}_{t_{p}} + \tilde{\Gamma}^{b}_{t_{p}} + \Theta^{t_{p}}_{0} + \tilde{\Theta}^{t_{p}}_{0}\big)^{-1} (\sigma^{m}-\sigma)\Vert_{H}\\
    &\leq \frac{1}{\alpha}\Vert (\mathcal{I}-\alpha(\alpha\mathcal{I}+\Gamma^{b}_{t_{p}} + \tilde{\Gamma}^{b}_{t_{p}} + \Theta^{t_{p}}_{0} + \tilde{\Theta}^{t_{p}}_{0})^{-1}\pi_{D})^{-1}  \Vert_{L(H)}\\
    &\times \Vert \alpha(\alpha\mathcal{I}+\Gamma^{b}_{t_{p}} + \tilde{\Gamma}^{b}_{t_{p}} + \Theta^{t_{p}}_{0} + \tilde{\Theta}^{t_{p}}_{0})^{-1} (\sigma^{m}-\sigma) \Vert_{H}\\
    &\leq \frac{1}{\alpha(1-\delta)}\Vert \sigma^{m}-\sigma\Vert_{H}\to 0,\, \text{as}\, m\to \infty.
\end{align*}
Then, we receive
\begin{align*}
    \Vert u^{m}_{\alpha}(s)-u_{\alpha}(s)\Vert_{U}&=\bigg\Vert \bigg( \sum_{k=1}^{p}\Omega^{*}S^{*}(t_{k}-s)\prod_{i=k+1}^{p}S^{*}(t_{i}-t_{i-1})\nonumber\\
    &\times S^{*}(b-t_{p})\chi_{(t_{k-1},t_{k})}+\Omega^{*}S^{*}(b-s)\chi_{(t_{p},b)}\bigg)\\
    &\times \big(\alpha (\mathcal{I}-\pi_{D})+\Gamma^{b}_{t_{p}} + \tilde{\Gamma}^{b}_{t_{p}} + \Theta^{t_{p}}_{0} + \tilde{\Theta}^{t_{p}}_{0}\big)^{-1}( \sigma^{m}-\sigma)\bigg\Vert_{U}\\
     &\leq \frac{\tilde{M}}{\alpha}\Vert \alpha\big(\alpha (\mathcal{I}-\pi_{D})+\Gamma^{b}_{t_{p}} + \tilde{\Gamma}^{b}_{t_{p}} + \Theta^{t_{p}}_{0} + \tilde{\Theta}^{t_{p}}_{0}\big)^{-1} ( \sigma^{m}-\sigma)\Vert_{H}\\
    &\leq \frac{\tilde{M}}{\alpha}\big\Vert (\mathcal{I}-\alpha (\alpha\mathcal{I}+\Gamma^{b}_{t_{p}} + \tilde{\Gamma}^{b}_{t_{p}} + \Theta^{t_{p}}_{0} + \tilde{\Theta}^{t_{p}}_{0})^{-1}\pi_{D})^{-1}\big\Vert_{L(H)}\\
    &\times \big\Vert \alpha(\alpha\mathcal{I}+\Gamma^{b}_{t_{p}} + \tilde{\Gamma}^{b}_{t_{p}} + \Theta^{t_{p}}_{0} + \tilde{\Theta}^{t_{p}}_{0})^{-1}( \sigma^{m}-\sigma)\big\Vert_{H}\\
     &\leq \frac{\tilde{M}}{\alpha}\frac{1}{1-\Vert \alpha (\alpha\mathcal{I}+\Gamma^{b}_{t_{p}} + \tilde{\Gamma}^{b}_{t_{p}} + \Theta^{t_{p}}_{0} + \tilde{\Theta}^{t_{p}}_{0})^{-1}\pi_{D}\Vert_{L(H)}}\Vert  \sigma^{m}-\sigma\Vert_{H}\\
     &\leq  \frac{\tilde{M}}{\alpha(1-\delta)}\Vert  \sigma^{m}-\sigma\Vert_{H}\to 0,\, \text{as}\, m\to \infty.
\end{align*}
For $t\in [0,t_{1}]$, we have
\begin{align*}
    \Vert (G_{\alpha}z^m)(t)-(G_{\alpha}z)(t)\Vert_{H}&\leq M_{S}M_{\Omega}\int_{0}^{t}\Vert u^{m}_{\alpha}(s)-u_{\alpha}(s)\Vert_{U}ds\\
    &+M_{S}\int_{0}^{t}\Vert \mu(s,z^{m}(s))-\mu(s,z(s))\Vert_{H}ds\\
    &\to 0,\,\, \text{as}\,\, m\to \infty.
\end{align*}
For $t\in [t_{k},t_{k+1}],\, k=1,\dots,p.$, we have
\begin{align*}
    \Vert z^{m}(t^{+}_{k})-z(t^{+}_{k})\Vert_{H}
    &=\bigg\Vert \sum_{i=1}^{k}\prod_{j=k}^{i+1}(\mathcal{I}+B_{j})S(t_{j}-t_{j-1})
(\mathcal{I}+B_{i})\int_{t_{i-1}}^{t_{i}}S(t_{i}-s)\Omega (u^{m}_{\alpha}(s)-u_{\alpha}(s))ds\\
&+ \sum_{i=1}^{k}\prod_{j=k}^{i+1}(\mathcal{I}+B_{j})S(t_{j}-t_{j-1})
(\mathcal{I}+B_{i})\int_{t_{i-1}}^{t_{i}}S(t_{i}-s)(\mu(s,z^{m}(s))-\mu(s,z(s)))ds\bigg\Vert_{H}\\
&\leq \sum_{m=1}^{k}(1+M_{B})^{m}M_{S}^{m}M_{\Omega}\int_{t_{m-1}}^{t_{m}} \Vert u^{m}_{\alpha}(s)-u_{\alpha}(s)\Vert_{U}ds\\
    &+ \sum_{m=1}^{k}(1+M_{B})^{m}M_{S}^{m}\int_{t_{m-1}}^{t_{m}} \Vert \mu(s,z^{m}(s))-\mu(s,z(s))\Vert_{H}ds\to 0,\, \text{as}\, m\to \infty.
\end{align*}
Then, we get
\begin{align*}
     \Vert (G_{\alpha}z^m)(t)-(G_{\alpha}z)(t)\Vert_{H}&\leq M_{S}\Vert z^{m}(t^{+}_{k})-z(t^{+}_{k})\Vert_{H} 
     +M_{S}M_{\Omega} \int_{t_{k}}^{t} \Vert u^{m}_{\alpha}(s)-u_{\alpha}(s)\Vert_{U} ds\Vert_{H}\\
      &+M_{S}\int_{0}^{t}\Vert \mu(s,z^{m}(s))-\mu(s,z(s))\Vert_{H}ds\to 0,\, \text{as}\, m\to \infty.
\end{align*}
Thus, it follows that \( G_{\alpha} \) is continuous. Consequently, by applying the SFPT, we conclude that the operator \( G_{\alpha} \) has a fixed point in \( \mathcal{B}_r \), which implies that the system \eqref{eq1} has a mild solution.   
\end{proof}
\begin{remark}\label{rem}
If Assumption \( (A_{2}) \) concerning the function \( \mu(\cdot, \cdot) \) is replaced with the following condition:  

\((A_{5})\) There exists a constant \( N > 0 \) such that 
\[ \|\mu(t, \nu)\|_H \leq N, \quad \text{for all } (t, \nu) \in [0,b] \times H, \]  
then the condition \( \eqref{f5} \) in the preceding theorem is automatically satisfied because \( d = 0 \).  
\end{remark}  
To examine the finite-approximate controllability of the semilinear system \( \eqref{eq1} \), we make the following assumption:

\((A_{6})\) The operator \( \alpha \Big(\alpha \mathcal{I} + \Gamma^{b}_{t_{p}} + \tilde{\Gamma}^{b}_{t_{p}} + \Theta^{t_{p}}_{0} + \tilde{\Theta}^{t_{p}}_{0}\Big)^{-1} \to 0 \) as \( \alpha \to 0^+ \) in the strong operator topology. Consequently, this implies that the linear system associated with \eqref{eq1} is approximately controllable on \( [0, b] \).
\bigskip

We now establish the following lemma, which is essential for proving the subsequent theorem.

\begin{lemma}\label{lem1}
    The set \(\{\mu(\cdot, z_\alpha(\cdot)) : \alpha > 0\}\) is bounded in \(L^2([r_1, r_2], H)\).
\end{lemma}  

\begin{proof}
Using Assumption (A3), we have  
\[
\|\mu(t, z)\|_H \leq g(t)\Lambda_\mu(\|z\|_H), \quad \forall t \in [r_1, r_2], \, z \in H,
\]  
where \(g \in L^\infty([r_1, r_2])\), and \(\Lambda_\mu\) is a continuous, non-decreasing function satisfying  
\[
\Lambda_\mu(r) \leq dr, \quad \forall r \geq 0,
\]  
for some constant \(d > 0\). Squaring this inequality, we get  
\[
\|\mu(t, z_\alpha(t))\|_H^2 \leq g(t)^2 \Lambda_\mu(\|z_\alpha(t)\|_H)^2 \leq d^2 g(t)^2 \|z_\alpha(t)\|_H^2.
\]

Integrating over \([r_1, r_2]\),  
\[
\int_{r_1}^{r_2} \|\mu(t, z_\alpha(t))\|_H^2 \, dt \leq d^2 \int_{r_1}^{r_2} g(t)^2  \|z_\alpha(t)\|_H^2\, dt.
\]
Since \(g \in L^\infty([r_1, r_2])\), there exists a constant \(C_g > 0\) such that \(|g(t)| \leq C_g\) a.e. \(t \in [r_1, r_2]\). This allows us to uniformly bound \(g(t)^2\) by \(C_g^2\). 

The condition \(z_\alpha \in \mathcal{B}_r\) implies that the set \(\{z_\alpha(t) : \alpha > 0\}\) is bounded in \(H\) for each \(t \in [r_1, r_2]\). Thus, there exists a constant \(r > 0\), independent of \(\alpha\) and \(t\), such that  
\[
\|z_\alpha(t)\|_H \leq r, \,  \forall t \in [r_1, r_2].
\]
Then, we have  
\[
\int_{r_1}^{r_2} \|\mu(t, z_\alpha(t))\|_H^2 \, dt \leq C.
\]
Define \(C = d^2 C_g^2 (r_2 - r_1)\). Since \(C\) is independent of \(\alpha\), the set \(\{\mu(\cdot, z_\alpha(\cdot)) : \alpha > 0\}\) is bounded in \(L^2([r_1, r_2], H)\). 
\end{proof}
\begin{theorem}\label{rrr8}
If assumptions \((A_1)\) and \((A_3)-(A_6)\) are satisfied, then the system \eqref{eq1} is $F_{A}$-controllable.
\end{theorem}
\begin{proof}
  By applying Theorem \ref{t1}, it follows that for any $\alpha > 0$ and $h \in H$, the system represented by \eqref{eq1} possesses a mild solution $z_{\alpha} \in B_{r}$, achieved through the control defined as follow:
\begin{align}\label{f55}
    u_{\alpha}(s)&=\bigg( \sum_{k=1}^{p}\Omega^{*}S^{*}(t_{k}-s)\prod_{i=k+1}^{p}S^{*}(t_{i}-t_{i-1})S^{*}(b-t_{p})\chi_{(t_{k-1},t_{k})}+\Omega^{*}S^{*}(b-s)\chi_{(t_{p},b)}\bigg)\Psi_{\alpha},
\end{align}
where 
\begin{align*}
    \Psi_{\alpha}&= \big(\alpha (\mathcal{I}-\pi_{D})+\Gamma^{b}_{t_{p}} + \tilde{\Gamma}^{b}_{t_{p}} + \Theta^{t_{p}}_{0} + \tilde{\Theta}^{t_{p}}_{0}\big)^{-1}\sigma_{\alpha},\\
    \sigma_{\alpha}&=h-S(b - t_p) \sum_{j=p}^1 (I + B_j) S(t_j - t_{j-1}) z_0
    -S(b - t_p)\sum_{i=1}^{p}\prod_{j=p}^{i+1}(\mathcal{I}+B_{j})S(t_{j}-t_{j-1})
	\\
     &\times  (\mathcal{I}+B_{i})\int_{t_{i-1}}^{t_{i}}S(t_{i}-s)\mu(s,z_{\alpha}(s))\,ds-\int_{t_{p}}^{b}S(b-s)\mu(s,z_{\alpha}(s))\, ds.
\end{align*}
 Consequently, it follows directly that
  \begin{align*}
  z_{\alpha}(t)=
      \begin{cases}
         S(t)z(0)+\int_{0}^{t} S(t-s)\big[\Omega u_{\alpha}(s)+\mu(s,z_{\alpha}(s))\big]ds,\,  0\leq t\leq t_{1},\\\\
S(t-t_{k})z_{\alpha}(t^{+}_{k})+\int_{t_{k}}^{t} S(t-s)\big[\Omega u_{\alpha}(s)+\mu(s,z_{\alpha}(s))\big]ds,\,
 t_{k}<t\leq t_{k+1},\ k=1,2,\dots,p, 
      \end{cases}
  \end{align*}
where  
\begin{align*}
z_{\alpha}(t^{+}_{k}) &= \prod_{j=k}^{1}(\mathcal{I}+B_{j})S(t_{j}-t_{j-1})z_{0}\\
&+ \sum_{i=1}^{k}\prod_{j=k}^{i+1}(\mathcal{I}+B_{j})S(t_{j}-t_{j-1})
(\mathcal{I}+B_{i})\int_{t_{i-1}}^{t_{i}}S(t_{i}-s)\Omega u_{\alpha}(s)ds\\
&+ \sum_{i=1}^{k}\prod_{j=k}^{i+1}(\mathcal{I}+B_{j})S(t_{j}-t_{j-1})
(\mathcal{I}+B_{i})\int_{t_{i-1}}^{t_{i}}S(t_{i}-s)\mu(s,z_{\alpha}(s))ds\\
& + \sum_{i=2}^{k}\prod_{j=k}^{i}(\mathcal{I}+B_{j}) S(t_{j}-t_{j-1}) D_{i-1}v_{i-1} + D_{k}v_{k}.
\end{align*}
Using the previously mentioned relationships, we obtain
\begin{align}\label{q123}
    z_{\alpha}(b)&=S(b-t_{k})z_{\alpha}(t^{+}_{k})+\int_{t_{k}}^{b} S(b-s)\big[\Omega u_{\alpha}(s)+\mu(s,z_{\alpha}(s))\big]ds\nonumber\\
&=h-\alpha(I-\pi_{D})\Big(\alpha (\mathcal{I}-\pi_{D})+\Gamma^{b}_{t_{p}} + \tilde{\Gamma}^{b}_{t_{p}} + \Theta^{t_{p}}_{0} + \tilde{\Theta}^{t_{p}}_{0}\Big)^{-1}\sigma_{\alpha}
\end{align}
The condition \( z_\alpha \in \mathcal{B}_r \) implies that the set \( \{z_\alpha(t) : \alpha > 0\} \) is bounded in \( H \) for each \( t \in [0, b] \). By applying the Banach–Alaoglu theorem, we can extract a subsequence \( \{z_{\alpha_i}(\cdot)\}_{i=1}^\infty \) such that,
\begin{align*}
    z_{\alpha_{i}}(t)\xrightarrow{w}z(t)\,\, in\,\, H\,\, \text{as}\,\, \alpha_{i}\to 0+\, (i\to \infty),\, \text{for all}  \,\, t\in [0,b].
\end{align*}
Moreover, by using Lemma \ref{lem1}, we obtain the set $\{ \mu(\cdot, z_{\alpha}(\cdot): \, \alpha>0 \}$ in $L^{2}([r_{1},r_{2}], H)$ is bounded. And, by applying the Banach-Alaoglu theorem, we can extract a subsequence $\{ \mu(\cdot, z_{\alpha_{i}}(\cdot)\}_{i=1}^{\infty}$, such that
\begin{align}\label{456}
    \mu(\cdot, z_{\alpha_{i}}(\cdot)) \xrightarrow{\text{w}}\mu(\cdot)\, \text{in} \, L^{2}([r_{1},r_{2}], H)\, \text{as}\, \alpha_{i}\to 0+\, (i\to\infty).
    \end{align}
By apply the Cauchy-Schwarz inequality, we have
\begin{align}\label{t678}
    \Vert \sigma_{\alpha_{i}}-\eta\Vert_{H}
     &\leq \sum_{m=0}^{p}M_{S}^{m}(1+M_{B})^{m} \bigg\Vert\int_{t_{p-m}}^{t_{p-m+1}}S(t_{p-m+1}-s)\big[\mu(s,z_{\alpha_{i}}(s))-\mu(s)\big]ds  \bigg\Vert_{H}\nonumber\\
     &\leq \sum_{m=0}^{p}M_{S}^{m+1}(1+M_{B})^{m}\sqrt{t_{p-m+1}-t_{p-m}} \bigg(\int_{t_{p-m}}^{t_{p-m+1}}\big\Vert \mu(s,z_{\alpha_{i}}(s))-\mu(s)\big\Vert^{2}_{H} \,ds \bigg)^{\frac{1}{2}} \nonumber\\
     &\to 0\, \, \text{as}\,\,  \alpha_{i}\to 0+ \,(i\to\infty),
\end{align}
where
\begin{align*}
   \eta&=h-S(b - t_p) \sum_{j=p}^1 (I + B_j) S(t_j - t_{j-1}) z_0
    -S(b - t_p)\sum_{i=1}^{p}\prod_{j=p}^{i+1}(\mathcal{I}+B_{j})S(t_{j}-t_{j-1})\\
	 &\times (\mathcal{I}+B_{i}) \int_{t_{i-1}}^{t_{i}}S(t_{i}-s)\mu(s)\,ds-\int_{t_{p}}^{b}S(b-s)\mu(s)\,ds.
\end{align*}
In \eqref{t678}, the convergences described in \eqref{456} are established using the Dominated Convergence Theorem. Furthermore, by utilizing \eqref{q123}, we obtain
\begin{align*}
    \Vert z_{\alpha_{i}}(b)-h\Vert_{H}&\leq \Big\Vert \alpha_{i}(I-\pi_{D})\Big(\alpha_{i} (\mathcal{I}-\pi_{D})+\Gamma^{b}_{t_{p}} + \tilde{\Gamma}^{b}_{t_{p}} + \Theta^{t_{p}}_{0} + \tilde{\Theta}^{t_{p}}_{0}\Big)^{-1}\sigma_{\alpha}\Big\Vert_{H}\\
    &\leq \Big\Vert \alpha_{i}(I-\pi_{D})\Big(\alpha_{i} (\mathcal{I}-\pi_{D})+\Gamma^{b}_{t_{p}} + \tilde{\Gamma}^{b}_{t_{p}} + \Theta^{t_{p}}_{0} + \tilde{\Theta}^{t_{p}}_{0}\Big)^{-1}\big(\sigma_{\alpha}-\eta\big)\Big\Vert_{H}\\
    &+ \Big\Vert \alpha_{i}(I-\pi_{D})\Big(\alpha_{i} (\mathcal{I}-\pi_{D})+\Gamma^{b}_{t_{p}} + \tilde{\Gamma}^{b}_{t_{p}} + \Theta^{t_{p}}_{0} + \tilde{\Theta}^{t_{p}}_{0}\Big)^{-1}\eta\Big\Vert_{H}\\
    &\leq \Vert I-\pi_{D} \Vert_{L(H)}\big\Vert \alpha_{i}\big(\alpha_{i} (\mathcal{I}-\pi_{D})+\Gamma^{b}_{t_{p}} + \tilde{\Gamma}^{b}_{t_{p}} + \Theta^{t_{p}}_{0} + \tilde{\Theta}^{t_{p}}_{0}\big)^{-1}\big\Vert_{L(H)}\\
     &\times  \Vert \sigma_{\alpha}-\eta\Vert_{H}+\Vert I-\pi_{D} \Vert_{L(H)}\big\Vert \alpha_{i}\big(\alpha_{i} (\mathcal{I}-\pi_{D})+\Gamma^{b}_{t_{p}} + \tilde{\Gamma}^{b}_{t_{p}} + \Theta^{t_{p}}_{0} + \tilde{\Theta}^{t_{p}}_{0}\big)^{-1}\eta \big\Vert_{H}\\
     &\leq \big\Vert \alpha_{i}\big(\alpha_{i} (\mathcal{I}-\pi_{D})+\Gamma^{b}_{t_{p}} + \tilde{\Gamma}^{b}_{t_{p}} + \Theta^{t_{p}}_{0} + \tilde{\Theta}^{t_{p}}_{0}\big)^{-1}\big\Vert_{L(H)}\Vert \sigma_{\alpha}-\eta\Vert_{H}\\
     &  +\big\Vert \alpha_{i}\big(\alpha_{i} (\mathcal{I}-\pi_{D})+\Gamma^{b}_{t_{p}} + \tilde{\Gamma}^{b}_{t_{p}} + \Theta^{t_{p}}_{0} + \tilde{\Theta}^{t_{p}}_{0}\big)^{-1}\eta \big\Vert_{H}\\
     &\leq \frac{\alpha_{i}}{\min(\alpha_{i},\mathscr{R})}\Vert \sigma_{\alpha}-\eta\Vert_{H} +\big\Vert \alpha_{i}\big(\alpha_{i} (\mathcal{I}-\pi_{D})+\Gamma^{b}_{t_{p}} + \tilde{\Gamma}^{b}_{t_{p}} + \Theta^{t_{p}}_{0} + \tilde{\Theta}^{t_{p}}_{0}\big)^{-1}\eta \big\Vert_{H}.
\end{align*}
Where $\mathscr{R}=\min\Big\{\Big\langle \pi_{D}\Big(\Gamma^{b}_{t_{p}} + \tilde{\Gamma}^{b}_{t_{p}} + \Theta^{t_{p}}_{0} + \tilde{\Theta}^{t_{p}}_{0}\Big) \pi_{D}\varphi, \varphi\Big\rangle\, : \, \Vert \pi_{D}\varphi\Vert_{H}=1 \Big\}>0.$ Using the convergence \eqref{t678}, Assumption $(A_{6})$ and Theorem \ref{t1}, we get 
\begin{align*}
    \Vert z_{\alpha_{i}}(b)-h\Vert_{H}\to 0\,\,  as\,\,  \alpha_{i}\to 0+\, (i\to 0).
\end{align*}
Moreover, by the estimate \eqref{q123}, we also get 
\begin{align*}
    \pi_{D}z_{\alpha_{i}}(b)=\pi_{D}h.
\end{align*}
Under this condition, the system described by \eqref{eq1} is $F_{A}$-controllable on \( [0, b] \).

\end{proof} 

\section{Application}

We study an impulsive evolution system connected to the heat equation:  

\begin{align}\label{kk1}
    \begin{cases}
        \frac{\partial \xi(t,z)}{\partial t} = \frac{\partial^2 \xi(t,z)}{\partial z^2} + \Omega u(t,z) + \mu(t, \xi(t,z)), & z \in (0, 1),\\
        \xi(t, 0) = \xi(t, 1) = 0, & t \in [0, b] \setminus \{t_1, \dots, t_p\},\\
        \xi(0, z) = \xi_0(z), & z \in [0, 1],\\
        \Delta \xi(t_k, z) = -\xi(t_k, z) - v_k(z), & z \in (0, 1),\, k = 1, \dots, p-1.
    \end{cases}
\end{align}

Consider the Hilbert space \( H = L^2([0, 1]) \) and the operator \( A: H \to H \), defined by \( A\xi = \xi^{\prime\prime} \). The domain of \( A \) is given by  
\[
D(A) = \{ w \in H : w \text{ and } w' \text{ are absolutely continuous},\, w'' \in H,\, w(0) = w(1) = 0 \}.
\]  

The operator \( A \) can also be expressed as  
\[
A w = -\sum_{n=1}^\infty n^2 \langle w, e_n \rangle e_n, \quad w \in D(A),
\]  

where the eigenvalues are \( \lambda_n = n^2 \), and the eigenfunctions are \( e_n(z) = \sqrt{2} \sin(nz) \), \( z \in [0, 1] \), for \( n = 1, 2, \dots \).  
 
Given an initial condition \( v_0 \in L^2([0, 1]) \), the impulsive term is specified as \( \Delta \xi(t_k, z) = -\xi(t_k, z) - v_k(z) \), where the operators \( B_k = D_k = -I \). Additionally, let \( \mu \)  satisfies Assumptions $(A_{2})$ and $(A_{3})$.  

It is well-known that the operator \( A \) generates a compact \(C_0\)-semigroup \( S(t) \) on \( H \), which can be represented as:  

\[
S(t) w = \sum_{n=1}^\infty e^{-n^2 t} \langle w, e_n \rangle e_n, \quad w \in H.
\]  

We define an infinite-dimensional control space \( U \) as:  

\[
U = \left\{ u : u = \sum_{n=2}^\infty u_n e_n, \, \sum_{n=2}^\infty u_n^2 < \infty \right\},
\]  

with the norm \( \|u\|_U = \sqrt{\sum_{n=2}^\infty u_n^2}. \)  

Next, we define the mapping \( \Omega: U \to H \) by 

\[
\Omega u = 2u_2 e_1 + \sum_{n=2}^\infty u_n e_n.
\]  

Since the semigroup \( S(t) \) generated by \( A \) is compact, the linear system associated with \eqref{kk1} does not achieve exact controllability. However, it satisfies \( A \)-controllability, as shown in \cite{Mahmudov2}. Consequently, system \eqref{kk1} can be reformulated in the abstract framework of equation \eqref{eq1}. Based on Theorem \ref{rrr8}, the system is \( F_A \)-controllable over the interval \([0, b]\).
\section{Conclusion}
In this study, we have analyzed impulsive evolution equations within the framework of Hilbert spaces.
First, by employing a resolvent-like operator, we established the $F_{A}-$ controllability for
linear systems. Subsequently, leveraging the SFPT, we demonstrated the existence
of solutions and verified the $F_{A}-$ controllability of semilinear impulsive systems in Hilbert
spaces. Finally, these results were extended to a broader setting, specifically to the heat equation, thereby
illustrating the applicability of our findings to more complex dynamical systems.

\end{document}